\documentclass[11pt]{article}
\usepackage{amssymb,amsmath,setspace,graphicx,multicol,wrapfig,amsfonts,amsthm,titling,url,array,parskip,enumerate,mathtools,scrextend}
\usepackage{hyperref}
\usepackage{placeins}
\usepackage{aurical,pbsi}
\usepackage[T1]{fontenc}
\usepackage[hmargin=2.5cm,vmargin=2.5cm]{geometry}

\makeatletter
\def\thm@space@setup{%
 \thm@preskip=0.1in
 \thm@postskip=0in
}
\makeatother

\numberwithin{equation}{section}

\theoremstyle{plain}
\newtheorem{thm}{Theorem}[section]
\newtheorem{lem}[thm]{Lemma}
\newtheorem{prop}[thm]{Proposition}

\theoremstyle{definition}
\newtheorem{defn}[thm]{Definition}

\theoremstyle{remark}
\newtheorem*{rem}{Remark}

\newtheorem*{example}{Example}

\DeclareMathOperator{\wt}{wt}
\DeclareMathOperator{\height}{ht}

\DeclareMathOperator{\fil}{fi}
\DeclareMathOperator{\one}{\mathbf{1}}
\DeclareMathOperator{\weight}{weight}
\DeclareMathOperator{\Prob}{Prob}

\DeclareMathOperator{\D}{D}
\DeclareMathOperator{\A}{A}
\DeclareMathOperator{\E}{E}
\DeclareMathOperator{\MCT}{MCT}

\begin{document}
\begin{center} {\Large{\sc Tableaux combinatorics of the two-species PASEP}} \\
\vspace{0.1in}
Olya Mandelshtam\footnote{\noindent University of California Berkeley, Phone: +1 (949) 689-5748, E-mail: olya@math.berkeley.edu} and Xavier Viennot\footnote{\noindent LaBRI, Universit\'{e} Bordeaux, Phone: +33 (0)5 4000-6085, 
E-mail: viennot@xavierviennot.org}
 \end{center}

\begin{abstract}
We study a two-species PASEP, in which there are two types of particles, ``heavy'' and ``light'', hopping right and left on a one-dimensional lattice of $n$ cells with open boundaries. In this process, only the ``heavy'' particles can enter on the left of the lattice and exit from the right of the lattice. In the bulk, any transition where a heavier particle type swaps places with an adjacent lighter particle type is possible. We generalize combinatorial results of Corteel and Williams for the ordinary PASEP by defining a combinatorial object which we call a rhombic alternative tableau that gives a combinatorial formula for the stationary probabilities for the states of this two-species PASEP.

\end{abstract}

\section{Introduction}

The asymmetric exclusion process (ASEP) is a model of interacting particles introduced in biology and mathematics in the late 1960's \cite{MGP68, Spi70}. The particles may hop to the right and left in a finite lattice of cells and enter and exit at the boundaries of the lattice, with the exclusion condition which means there can be at most one particle in each cell. In the most general model, five parameters determine the hopping rates for the particles. In the bulk, the rate of hopping left is $q$ times the rate of hopping right. The four other parameters refer to the rates at which particles enter and exit at the left and right boundaries of the lattice. This model has attracted considerable attention in physics, and in particular is treated as a toy model in the study of dynamical systems far from thermal equilibrium. The ASEP can be solved completely in the general case (i.e.\ there are explicit formulas for the stationary probabilities) and it exhibits boundary-induced phase transitions. A seminal tool is the so-called Matrix Ansatz of Derrida, Evans, Hakim and Pasquier \cite{DEHP93}.

In the past ten years, much work has been done in combinatorics in relation to the ASEP. Most notably, the stationary probabilities can be interpreted as a weighted sum over certain combinatorial objects. The main starting point for the combinatorial interpretations was the ASEP with three parameters $\alpha$, $\beta$ and $q$, i.e.\ the partially asymmetric exclusion process (PASEP) (this is the same model as the ASEP, except that some parameters are set to 0). In this process, particles can enter the lattice on the left with probability proportional to $\alpha$ and exit at the right with probability proportional to $\beta$, and hop right or left inside the strip with probabilities proportional to 1 and $q$ respectively. In relation with the Matrix Ansatz \cite{DEHP93}, various authors have introduced several equivalent combinatorial objects such as \emph{permutation tableaux, alternative tableaux, staircase tableaux, and tree-like tableaux}, all enumerated by $n!$ and thus in bijection with permutations. In the case where $q=0$ (TASEP), the corresponding tableaux are enumerated by Catalan numbers. A very rich area of combinatorics has emerged from this topic. 

%
%

The PASEP has also been generalized to allow different species of particles, with some priority rules where particles can swap with particles of different species. Again, some combinatorial interpretations are starting to be discovered, using an extended Matrix Ansatz. In this paper we consider a simple two-species PASEP with three parameters $\alpha$, $\beta$ and $q$, analogous to the ordinary PASEP (see \cite{Uch07, ALS09, DS05}). In \cite{M16}, the first author studied the case of the two-species PASEP for $q=0$, and introduced an object called the ``multi-Catalan tableau'' that gives an interpretation for the steady state probabilities of the process for that special case. In this work, we generalize the result for all $q$ with a new object called the rhombic alternative tableaux. 

More specifically, the two-species PASEP has two species of particles, one heavy and one light. The ``heavy'' particle can enter the lattice on the left with rate $\alpha$, and exit the lattice on the right with rate $\beta$. Moreover, the ``heavy'' particle can swap places with both the ``hole'' and the ``light'' particle when they are adjacent, and similarly the ``light'' particle can swap places with the ``hole''. Each of these possible swaps occur at rate 1 when the heavier particle is to the left of the lighter one, and at rate $q$ when the heavier particle is to the right (we simplify our notation by treating the ``hole'' as a third type of ``particle''). More precisely, if we denote the ``heavy'' particle by $\D$, the ``light'' particle by $\A$, and the ``hole'' by $\E$, and let $X$ and $Y$ be any words in $\{\D, \A, \E\}$ then the transitions of this process are:
\begin{displaymath} X\A\E Y \overset{1}{\underset{q}{\rightleftharpoons}} X\E\A Y \qquad X\D\E Y \overset{1}{\underset{q}{\rightleftharpoons}} X\E\D Y \qquad X\D\A Y \overset{1}{\underset{q}{\rightleftharpoons}} X\A\D Y\end{displaymath}

\begin{displaymath}\E X \overset{\alpha}{\rightharpoonup} \D X \qquad \qquad X\D \overset{\beta}{\rightharpoonup} X\E\end{displaymath}
where by $X \overset{u}{\rightharpoonup} Y$ we mean that the transition from $X$ to $Y$ has probability $\frac{u}{n+1}$, $n$ being the length of $X$ (and also $Y$).

Notice that since only the ``heavy'' particle can enter or exit the lattice, the number of ``light'' particles must stay fixed. In particular, if we fix the number of ``light'' particles to be 0, we recover the original PASEP. 

Uchiyama provides a Matrix Ansatz along with matrices that satisfy the conditions, to express the stationary probabilities of the two-species PASEP as certain matrix products.

\begin{thm}[Uchiyama \cite{Uch07}]\label{ansatz}
Let $W = \{W_1,\ldots,W_n\}$ with $W_i \in \{\D, \A, \E\}$ for $1 \leq i \leq n$ represent a state of the two-species PASEP of length $n$ with $r$ ``light'' particles. Suppose there are matrices $D$, $E$, and $A$ and vectors $\langle w|$ and $|v \rangle$ which satisfy the following conditions
\[
DE = D+E+qED \qquad DA = A + qAD \qquad AE = A + qEA
\]
\[
\langle w| E = \frac{1}{\alpha} \langle w| \qquad D|v \rangle= \frac{1}{\beta} |v \rangle
\]
then 
\[
\Prob(W) = \frac{1}{Z_{n,r}} \langle w| \prod_{i=1}^n D\one_{(W_i=\D)} + A\one_{(W_i=\A)} + E\one_{(W_i=\E)}|v \rangle
\]
where $Z_{n,r}$ is the coefficient of $y^r$ in $\langle w| (D + yA + E)^n|v \rangle$.
\end{thm}

This result generalizes a previous Matrix Ansatz solution for the regular PASEP of Derrida et. al. in \cite{DEHP93}. Inspired by Uchimaya's Matrix Ansatz, we introduce in this paper rhombic alternative tableaux (RAT). They are defined in Section \ref{tat_def}, but we state our main theorem below.

\begin{thm}\label{mainresult}
Let $W$ be a state of the two-species PASEP of size $n$ with exactly $r$ ``light'' particles. Then the stationary probability of state $W$ is
\[
\Pr(W) = \frac{1}{\mathcal{Z}_{n,r}} \sum_{T} \wt(T)
\]
where $T$ ranges over the rhombic alternative tableaux corresponding to $W$, $\wt(T)$ is the weight of such a tableau, and $\mathcal{Z}_{n,r}$ is the weight generating function for the set of rhombic alternative tableaux corresponding to the state space of $W$.
\end{thm} 

For the rest of the paper, in Section \ref{tat_def}, we introduce the rhombic alternative tableaux, and in Section \ref{ansatz_sec} we prove Theorem \ref{mainresult}. 
In Section \ref{q0_sec}, we conclude with a discussion of a special case for the rhombic alternative tableaux when $q=0$.

\section{Rhombic alternative tableaux}\label{tat_def}

\begin{wrapfigure}[4]{r}{0.08\textwidth}
\centering
\includegraphics[width=0.08\textwidth]{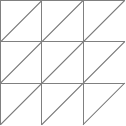}
\noindent
\label{lattice}
\end{wrapfigure}
The rhombic alternative tableaux (RAT) are an analog on a ``triangular lattice'' of the alternative tableaux \cite{slides} that correspond to the ordinary PASEP. By triangular lattice, we mean one which has as its vertices the integer points $(i,j)$, and the possible edges are the south edges with vertices $\{(i,j),(i,j-1)\}$, west edges with vertices $\{(i,j),(i-1,j)\}$, and southwest edges with vertices $\{(i,j),(i-1,j-1)\}$ for integers $i,j$, as in the figure on the right. 


\subsection{Definition of RAT}
\begin{defn}
Let $W$ be a word in the letters $\{\D, \A, \E\}$ with $k$ $\D$'s, $\ell$ $\E$'s, and $r$ $\A$'s of total length $n\vcentcolon= k+\ell+r$. Define $P_1$ to be the path obtained by reading $W$ from left to right and drawing a south edge for a $\D$, a west edge for an $\E$, and a southwest edge for an $\A$. (From here on, we call any south edge a D-edge, any west edge an E-edge, and any southwest edge an A-edge.) Define $P_2$ to be the path obtained by drawing $\ell$ west edges followed by $r$ southwest edges, followed by $k$ south edges. A \textbf{rhombic diagram} $\Gamma(W)$ of \textbf{type} $W$ is a closed shape on the triangular lattice that is identified with the region obtained by joining the northeast and southwest endpoints of the paths $P_1$ and $P_2$ (see Figure \ref{minimal}).
\end{defn}

\begin{figure}[h]
\begin{minipage}{0.4\textwidth}
\centering
\includegraphics[width=0.6\textwidth]{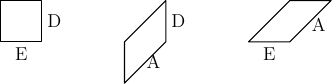}
\caption{The tiles DE, DA, and AE.}\label{tiles}
\end{minipage}\hfill
\begin{minipage}{0.6\textwidth}
\centering
\includegraphics[width=0.4\textwidth]{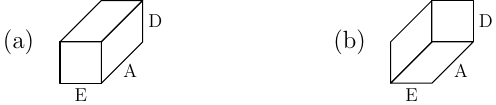}
\caption{(a) maximal and (b) minimal hexagons.}\label{hexagons}
\end{minipage}\hfill
\end{figure}

\begin{defn}
A \textbf{tiling} $\mathcal{T}$ of a rhombic diagram is a collection of open regions of the following three parallelogram shapes as seen in Figure \ref{tiles}, the closure of which covers the diagram:
\begin{itemize}
\item A parallelogram with south and west edges which we call a DE tile.
\item A parallelogram with southwest and west edges which we call an AE tile.
\item A parallelogram with south and southwest edges which we call a DA tile.
\end{itemize}
We define the \textbf{area} of a tiling to be the total number of tiles it contains.
\end{defn}

By convention, we label the E-edges of the southeast boundary of the rhombic diagram with 1 through $\ell$ from right to left, and the D-edges with 1 through $k$ from top to bottom.

\begin{defn}
A \textbf{north-strip} on a rhombic diagram with a tiling is a maximal strip of adjacent tiles of types DE or AE, where the edge of adjacency is always an E-edge. A \textbf{west-strip} is a maximal strip of adjacent tiles of types DE or DA, where the edge of adjacency is always a D-edge. The $i$'th north-strip is the north-strip whose bottom-most edge is the $i$'th (from right to left) E-edge on the boundary of the rhombic diagram. The $j$'th west-strip is the west-strip whose right-most edge is the $j$'th (from top to bottom) D-edge on the boundary of the rhombic diagram. Figure \ref{strips} shows an example of the west- and north-strips.
\end{defn}

Note that the number of tiles in the $i$'th north-strip is the total number of $\D$'s and $\A$'s in the word $W$ preceding the $i$'th $\E$. Similarly, the number of tiles in the $j$'th west-strip is the total number of $\E$'s and $\A$'s in the word $W$ following the $j$'th D. 


\begin{figure}[h]
\centering
\includegraphics[width=0.5\textwidth]{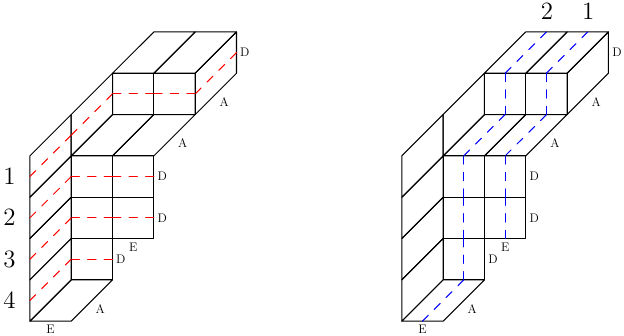}
\caption{\emph{Left)} west-strips and \emph{(right)} north-strips.}
\noindent
\label{strips}
\end{figure}

Finally we define the main object we are working with.

\begin{figure}
\begin{minipage}{0.48\textwidth}
\centering
\includegraphics[width=0.4\textwidth]{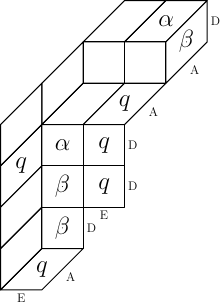}
\caption{An example of a RAT of size $(9,3,4)$ with type $\D\A\A\D\D\E\D\A\E$ and weight $\alpha^6\beta^5q^4$.}
\noindent
\label{RAT_example}
\end{minipage}\hfill
\begin{minipage}{0.48\textwidth}
\centering
\includegraphics[width=0.4\textwidth]{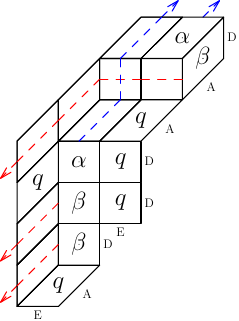}
\caption{A complete representation of a RAT that is equivalent to the example on the left.}
\noindent
\label{RAT_example1}
\end{minipage}\hfill
\end{figure}

\begin{defn}\label{RAT_def}
A RAT of \textbf{type} $W$ is a rhombic diagram $\Gamma(W)$ and an arbitrary tiling $\mathcal{T}$ with DE, DA, and AE tiles, and a filling $F$ of $\mathcal{T}$ with $\alpha$'s and $\beta$'s with the following conditions: 
\begin{enumerate}[i.]
\item A DE tile is empty or contains an $\alpha$ or a $\beta$.
\item A DA tile is empty or contains a $\beta$.
\item An AE tile is empty or contains an $\alpha$.
\item Any tile above and in the same north-strip as an $\alpha$ must be empty.
\item Any tile to the left and in the same west-strip as a $\beta$ must be empty.
\end{enumerate} 
\end{defn}

We define $\fil(W,\mathcal{T})$ to be the set of fillings of tiling $\mathcal{T}$ of the rhombic diagram $\Gamma(W)$. In other words, $F \in \fil(W,\mathcal{T})$ means $F$ is a filling of type $W$ of the tiling $\mathcal{T}$.

\begin{wrapfigure}[15]{r!}{0.42\textwidth}
\centering
\includegraphics[width=0.23\textwidth]{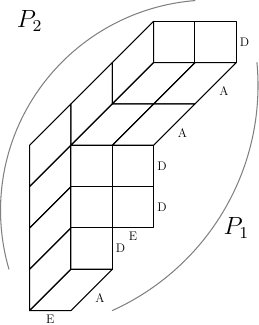}
\caption{The minimal tiling of a rhombic diagram $\Gamma(\D\A\A\E\E\D\E\A\E)$ defined by the paths $P_1$ and $P_2$.}
\noindent
\label{minimal}
\end{wrapfigure}

\begin{defn}
A \textbf{north line} is a line drawn through each north-strip containing an $\alpha$, starting at the tile directly above that $\alpha$. A \textbf{west line} is a line drawn through each west-strip containing a $\beta$, starting at the tile directly left of that $\beta$. An example of the RAT with the north- and west lines is shown in Figure \ref{RAT_example1}.
\end{defn}

In terms of the north- and west lines, we rewrite the conditions (iv) and (v) of Definition \ref{RAT_def} by (equivalently) requiring that any tile that contains a north line or a west line must be empty.

\begin{defn}
The \textbf{size} of a RAT of type $W$ is $(n,r,k)$, where $k$ is the number of $\D$'s in $W$, $r$ is the number of $\A$'s in $W$, and $n$ is the total number of letters in $W$. We can also call this the size of a filling $F$ of type $W$. We can also refer to the size of a tableau as simply $(n,r)$, where we do not keep track of the number of $\D$'s.
\end{defn}

\begin{defn}\label{weight}
To compute the \textbf{weight} $\wt(F)$ of a filling $F$, first a $q$ is placed in every empty tile that does not contain a north line or a west line. Next, $\wt(F)$ is the product of all the symbols inside $F$ times $\alpha^k \beta^{\ell}$, for $F$ a filling of size $(k+\ell+r,r,k)$.
\end{defn}

We will prove in Proposition \ref{equiv} that the sum of the weights of all fillings of $\Gamma(W)$ does not depend on the tiling $\mathcal{T}$.

\subsection{Independence of tilings and definition of $\weight(W)$}

\begin{prop}\label{equiv}
Let $W$ be a word in $\{\D, \A, \E\}$. Let $\mathcal{T}_1$ and $\mathcal{T}_2$ represent two different tilings of a rhombic diagram $\Gamma(W)$ with DE, DA, and AE tiles. Then 
\[
\sum_{F \in \fil(W,\mathcal{T}_1)} \wt(F)=\sum_{F' \in \fil(W,\mathcal{T}_2)} \wt(F').
\]
\end{prop}

\begin{wrapfigure}[5]{r}{0.4\textwidth}
\centering
\includegraphics[width=0.2\textwidth]{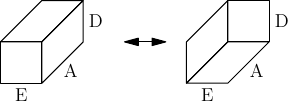}
\caption{A flip from a maximal (left) to a minimal hexagon (right).}
\noindent
\label{flip}
\end{wrapfigure}

\begin{defn} Consider a hexagon with vertices $\{(i,j), (i,j-1), (i-1,j-2), (i-2,j-2), (i-2,j-1), (i-1,j)\}$ for some integers $i,j$ that is tiled with a DE-, a DA-, and an AE tile. A \textbf{minimal hexagon} is when the tiles have the configuration of Figure \ref{hexagons} (b), and a \textbf{maximal hexagon} is when the tiles have the configuration of Figure \ref{hexagons} (a).
\end{defn}

\begin{defn}
Let $W$ be a word in $\{\D, \A, \E\}$. We define the \textbf{minimal tiling} of $\Gamma(W)$ to be the tiling that that does not contain an instance of a maximal hexagon, such as the example in Figure \ref{minimal}. We refer to such a tiling by $\mathcal{T}_{min}$\footnote{In the remark following the proof of Lemma \ref{flip_connectivity}, we show that $\mathcal{T}_{min}$ is the unique minimal tiling.}. $\mathcal{T}_{min}$ can be constructed by placing tiles from $P_1$ inwards, and placing an AE tile with first priority whenever possible.  (Similarly, a \textbf{maximal tiling} is one that that does not contain an instance of a minimal hexagon, and is referred to by $\mathcal{T}_{max}$. The maximal tiling can be constructed by placing tiles from $P_1$ inwards with priority given to the DA tiles whenever possible.)
\end{defn}

\begin{defn} A \textbf{flip} is an involution that switches between a maximal hexagon and a minimal hexagon, and is the particular rotation of tiles that is shown in Figure \ref{flip}.
\end{defn}

The lemma below contains a generally known result, notably in the case of a plane partition. 

\begin{lem}\label{flip_connectivity}
Let $\Gamma(W)$ be a rhombic diagram of type $W$. For any two tilings $\mathcal{T}$ and $\mathcal{S}$ of $\Gamma(W)$, $\mathcal{T}$ can be obtained from $\mathcal{S}$ by some series of flips.
\end{lem}

\vspace{0.2in}
\begin{rem}
To make the paper self-contained, we will give here a proof that tilings are in bijection with configurations of non-crossing paths, a well-known fact. Our proof defines a classical construction, in particular in the case of a plane partition where the bijection will give a configuration of paths related to a binomial determinant (by the LGV lemma), which can be expressed by a simple formula giving the well-known MacMahon formula for plane partitions (or 3D Ferrers diagrams) within a box of size $(a,b,c)$. We note that the case of a plane partition within a box of size $(k,r,\ell)$ is equivalent to the tilings of $\Gamma(W)$ where $W = D^kA^rE^{\ell}$.
\end{rem}

\begin{defn} A \textbf{non-crossing configuration} $\{\tau\}$ is a collection of lattice paths $\{\tau_1,\tau_2,\ldots\}$ where no two paths share a vertex (and hence are non-crossing).
\end{defn}

\begin{proof}[Proof of \ref{flip_connectivity}]
We prove the lemma by constructing a bijection from certain non-crossing configurations on $\Gamma(W)$ to tilings of $\Gamma(W)$. We define a height function $\height(\{\tau\})$ on the non-crossing configurations. We define the \emph{minimal} non-crossing configuration $\{\tau\}_{min}$ to be the one with height 0. We also define a local move called a \emph{slide} on $\{\tau\}$ which we show corresponds to a flip on $\mathcal{T}$ (similarly, a reverse slide corresponds to a flip in the other direction). We show that a slide diminishes the height of a non-crossing configuration by 1, and that any non-crossing configuration with height greater than 0 admits a slide. Thus it is possible to obtain $\{\tau\}_{min}$ from any $\{\tau\}$ by some set of slides, and so it is possible to obtain any $\{\sigma\}$ from any $\{\tau\}$ by some set of slides and reverse slides. This translates to the desired result due to the bijection.

\begin{figure}[h]
\centering
\includegraphics[width=\textwidth]{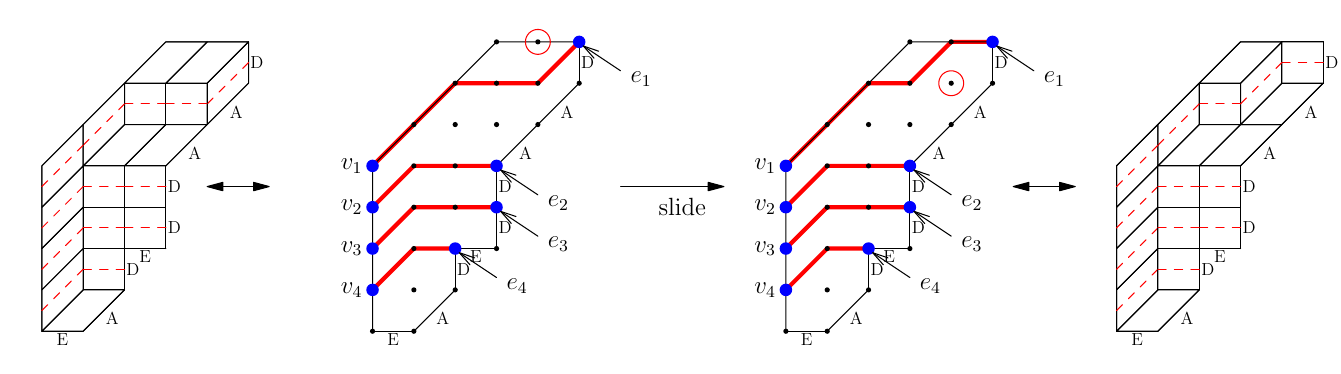}
\caption{\emph{(Left)} a tiling $\mathcal{T}$ with west-strips indicated, bijection to the non-crossing configuration $\{\tau\}$ with $\height(\{\tau\})=2$. \emph{(Right)} a slide performed at the indicated free lattice point to obtain the non-crossing configuration $\{\tau'\}$ with $\height(\{\tau'\})=1$, which is in bijection with a different tiling $\mathcal{T}'$ (also obtained from $\mathcal{T}$ by performing a flip at the northeast-most hexagon).}
\noindent
\label{slide}
\end{figure}

\noindent \textit{Bijection from tilings on $\Gamma(W)$ to non-crossing configurations on $\Gamma(W)$.} Recall that $\Gamma(W)$ is on a lattice with integer points $\{(i,j)\}$, and a tiling on $\Gamma(W)$ consists of tiles DE, DA, and AE whose corners are on those lattice points. We define a bijection from a tiling $\mathcal{T}$ to a non-crossing configuration $\{\tau\}$ on the lattice contained within $\Gamma(W)$. Let $\{\tau\} = \{\tau_i,\ldots,\tau_k\}$, where each $\tau_i$ is a lattice path consisting of west and southwest steps $(j,\ell)\rightarrow(j-1,\ell)$ and $(j,\ell)\rightarrow(j-1,\ell-1)$ respectively, that starts at $e_i$ and ends at $v_i$ for some sets $\{e_i\}$ and $\{v_i\}$. We let $e_i$ be the north endpoint of the $i$'th D-edge on the southeast boundary of $\Gamma(W)$ (from top to bottom), and let $v_i$ be the north endpoints of the $i$'th D-edge on the northwest boundary of $\Gamma(W)$ (from top to bottom). 

To obtain $\{\tau\}$ from $\mathcal{T}$, let $\tau_i$ be the path that coincides with the northwest boundary of the $i$'th west-strip (from top to bottom). Since the west-strips do not cross each other, it is clear that $\{\tau\}$ is well defined in this way (see Figure \ref{slide}).  

Define a \emph{free lattice point} to be a lattice point that is not part of any path in $\{\tau\}$, and does not lie on the southeast boundary of $\Gamma(W)$. To obtain $\mathcal{T}$ from $\{\tau\}$, we do the following: for each $\tau_i$, place a DE tile directly below and adjacent to each west step of $\tau_i$, and place a DA tile directly below and adjacent to each southwest step, so that the northwest boundaries of the DE and DA tiles coincide with the steps of $\tau_i$. Thus $\tau_i$ corresponds to a west-strip. Now, for every free lattice point that is not on the southeast boundary of $\Gamma(W)$, place an AE tile so that its northwest corner coincides with that lattice point. We claim that we obtain in this way a valid tiling of $\Gamma(W)$. To check this, we must simply verify that the construction above results in no tiles overlapping, and a complete covering of the shape. 

\noindent \textit{Height of $\{\tau\}$.}
For each $i$, define the \emph{minimal path} $m_i$ to be the one starting at $e_i$ and taking a maximal possible number of west steps followed by a maximal possible number of southwest steps to $v_i$. Define the \emph{height} of $\tau_i \in \{\tau\}$ (i.e.\ $\height(\tau_i)$) to be the area between $\tau_i$ and $m_i$ (i.e.\ the number of free lattice points strictly northwest of $\tau_i$ and weakly southeast of $m_i$). We define the height of $\{\tau\}$ to be
\[
\height(\{\tau\}) = \sum_{\tau_i \in \{\tau\}} \height(\tau_i).
\]
It is clear that there is a unique $\{\tau\}$ of height 0, by letting each $\tau_i$ be $m_i$. We call this $\{\tau_{min}\}$. 

\noindent \textit{A slide on $\{\tau\}$.}
Let $p$ be a free lattice point $(i,j)$ such that $(i+1,j)$, $(i,j-1)$, and $(i-1,j-1)$ are all not free lattice points. Then, those lattice points must necessarily belong to the same path in the non-crossing configuration $\{\tau\}$, say $\tau_i$. A \textbf{slide} on $\{\tau\}$ at the location of $p$ means exchanging the steps southwest and west for the steps west and southwest in $\tau_i$, and thereby passing through $p$ and creating a new free lattice point below $\tau_i$, as in Figure \ref{slide}. More precisely, the steps $(i+1,j)\rightarrow(i,j-1)\rightarrow(i-1,j-1)$ in $\tau_i$ are exchanged for $(i+1,j)\rightarrow(i,j)\rightarrow(i-1,j-1)$ to make $\tau'_i$. Clearly $\tau'_i$ does not cross $\{\tau_1,\ldots,\tau_{i-1},\tau_{i+1},\ldots\}$, and so the new collection of paths $\{\tau'\}$ formed by replacing $\tau_i$ with $\tau'_i$ is also a non-crossing configuration. Furthermore, since $\height(\tau_i') = \height(\tau_i)-1$, we have $\height(\{\tau'\}) = \height(\{\tau\})-1$.

We have already established that a southwest step of $\tau_i$ in $\{\tau\}$ corresponds to a DA box in the west-strip $i$, and a west step of $\tau_i$ corresponds to a DE box. A free lattice point necessarily corresponds to an AE tile, since all DA and DE tiles must be part of some west-strips. From Figure \ref{slide}, it is easy to see how a slide corresponds to a flip from a maximal hexagon to a minimal hexagon, and a reverse slide (the inverse operation) corresponds to the reverse flip.

Notice that if no such free lattice point $p$ exists such that its three neighbors east, south, and southwest are all not free lattice points, this implies $\tau_i=m_i$ for each $i$. Then $\{\tau\}$ is the minimal non-crossing configuration, as in Figure \ref{min_paths}. Thus $\{\tau\}$ admits no slides if and only if it equals $\{\tau_{min}\}$.

Now we complete the proof. Let $\{\tau\}$ be some non-crossing configuration with $\height(\{\tau\})=k>0$. Then by the above there is at least one free lattice point that admits a slide. After performing a slide at that location, we obtain a new non-crossing configuration $\{\tau'\}$ with $\height(\{\tau'\})=\height(\{\tau\})-1$ . Recursively, this implies that by applying some series of slides, we can get from $\{\tau\}$ to a non-crossing configuration with height 0. However, $\{\tau_{min}\}$ is the unique such non-crossing configuration, so we have shown that we can get from $\{\tau\}$ to $\{\tau_{min}\}$ with some set of slides. We can now equivalently define $\height(\{\tau\})$ as the minimal number of slides required to get from $\{\tau\}$ to $\{\tau_{min}\}$.

Let $\{\sigma\}$ be a different set of non-crossing configurations. There is similarly a set of slides to get from $\{\sigma\}$ to $\{\tau_{min}\}$. Thus we can get from $\{\tau\}$ to $\{\sigma\}$ by a series of slides, by first applying the slides to get from $\{\tau\}$ to $\{\tau_{min}\}$, and then by applying slides in reverse to get from $\{\tau_{min}\}$ to $\{\sigma\}$. Let the tiling $\mathcal{T}$ correspond to the non-crossing configuration $\{\tau\}$, and the tiling $\mathcal{S}$ to the non-crossing configuration $\{\sigma\}$. Since the slides on the paths correspond to flips on the tilings, we obtain that one can get from $\mathcal{T}$ to $\mathcal{S}$ with a series of flips, as desired. 
\end{proof}

\begin{rem} 
It is easy to check that $\{\tau_{min}\}$ corresponds to the minimal tiling $\mathcal{T}_{min}$ of $\Gamma(W)$ as in Figure \ref{min_paths}, since $\mathcal{T}_{min}$ is a tiling that admits zero flips from a maximal hexagon to a minimal hexagon, which corresponds to a non-crossing configuration admitting zero slides and having height 0. Since $\{\tau_{min}\}$ is the unique non-crossing configuration of height 0, $\mathcal{T}_{min}$ must be the unique minimal tiling according to our definition of ``minimal''. (The maximal tiling $\mathcal{T}_{max}$ is also unique by a similar argument.) Furthermore, the minimal number of flips required to get from $\mathcal{T}$ to $\mathcal{T}_{min}$ is commonly referred to as the height of a tiling $\mathcal{T}$.
\end{rem}

\begin{wrapfigure}[10]{r!}{0.5\textwidth}
\centering
\includegraphics[width=0.38\textwidth]{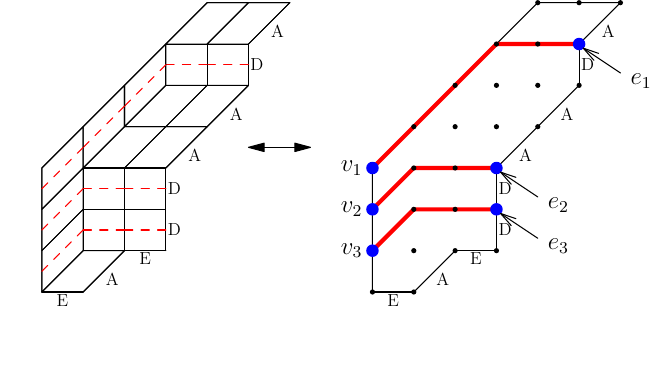}
\caption{$\mathcal{T}_{min}$ and $\{\tau_{min}\}$.}
\noindent
\label{min_paths}
\end{wrapfigure}


For the proof of Lemma \ref{equiv}, we introduce a more explicit set of tiles, where tiles can now contain $\alpha$, $\beta$, $q$, a north line, a west line, or both a north line and a west line, as in Figure \ref{tiles2}. We can now describe a \emph{complete} covering of $S$ with the following.

\begin{figure}[t]
\centering
\includegraphics[width=\textwidth]{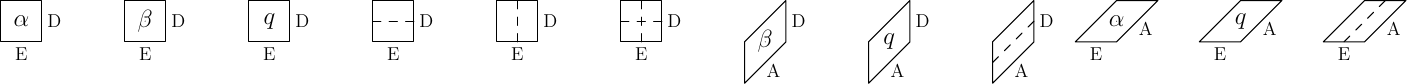}
\caption{A more explicit set of tiles that is in simple correspondence with the RAT fillings.}
\noindent
\label{tiles2}
\end{figure}

\begin{defn}\label{compatibility}
Two adjacent tiles are \textbf{compatible} if: 
\begin{enumerate}
\item a tile has a north line through it if and only if its south E-edge is adjacent to a tile containing an $\alpha$ or a north line, and
\item a tile has a west line through it if and only if its east D-edge is adjacent to a tile containing a $\beta$ or a west line.
\end{enumerate}
\end{defn}

\begin{proof}[Proof of \ref{equiv}]
For any rhombic diagram $\Gamma(W)$, any two tilings $\mathcal{T}_1$ and $\mathcal{T}_2$ can be obtained from each other by some series of flips by Lemma \ref{flip_connectivity}. Thus it is sufficient to show that if $\mathcal{T}_1$ and $\mathcal{T}_2$ differ by a single flip, then there is a weight-preserving bijection between $\fil(W,\mathcal{T}_1)$ and $\fil(W,\mathcal{T}_2)$. Let this flip occur at a certain ``special hexagon'' ($\mathfrak{h}_1$ in $\mathcal{T}_1$ and $\mathfrak{h}_2$ in $\mathcal{T}_2$). Without loss of generality, let $\mathfrak{h}_1$ be of minimal type as on the left of Figure \ref{flip}, and let $\mathfrak{h}_2$ be of maximal type as on the right of Figure \ref{flip}. The rest of the tiles are identical in $\mathcal{T}_1$ and $\mathcal{T}_2$. 

We define the bijection from $F \in \fil(W,\mathcal{T}_1)$ to some $F' \in \fil(W,\mathcal{T}_2)$ with an involution $\phi$. To begin, $\phi$ sends every tile including its contents in $\mathcal{T}_1\backslash\mathfrak{h}_1$ to its identical copy in $\mathcal{T}_2\backslash\mathfrak{h}_2$. Then, $\phi$ sends the tiles and contents of $\mathfrak{h}_1$ to a rearrangement of those tiles according to the cases shown in Figure \ref{flip_map}. It is easy to see that this map preserves the weights of the fillings, since the quantities of $\beta$s, $\alpha$'s, and $q$'s are preserved for each case. 

We claim that the map $\phi$ also preserves the compatibility of the tiles, as defined in Definition \ref{compatibility}. It is easy to check that in each possible case of $\mathfrak{h}_1$, the tile adjacent to the south E-edge contains an $\alpha$ or a north line if and only if the tile adjacent to the south E-edge of $\phi(\mathfrak{h}_1)$ contains an $\alpha$ or a north line. Similarly, the tile adjacent to the east D-edge of $\mathfrak{h}_1$ contains a $\beta$ or a west line if and only if the tile adjacent to the east D-edge of $\phi(\mathfrak{h}_1)$ contains a $\beta$ or a west line.

Thus $\phi$ indeed gives a weight-preserving bijection from $\fil(W,\mathcal{T}_1)$ to $\fil(W,\mathcal{T}_2)$, and so the lemma follows.
\end{proof}

Thus we are able to make the following definition.

\begin{defn}\label{weight_word}
Let $W$ be a word in $\{\D, \A, \E\}$, and let $\mathcal{T}$ be an arbitrary tiling of $\Gamma(W)$. Then the \textbf{weight} of a word $W$ is 
\[
\weight(W)=\sum_{F \in \fil(W,\mathcal{T})} \wt(F).
\]
\end{defn}

\begin{figure}[h]
\centering
\includegraphics[width=\textwidth]{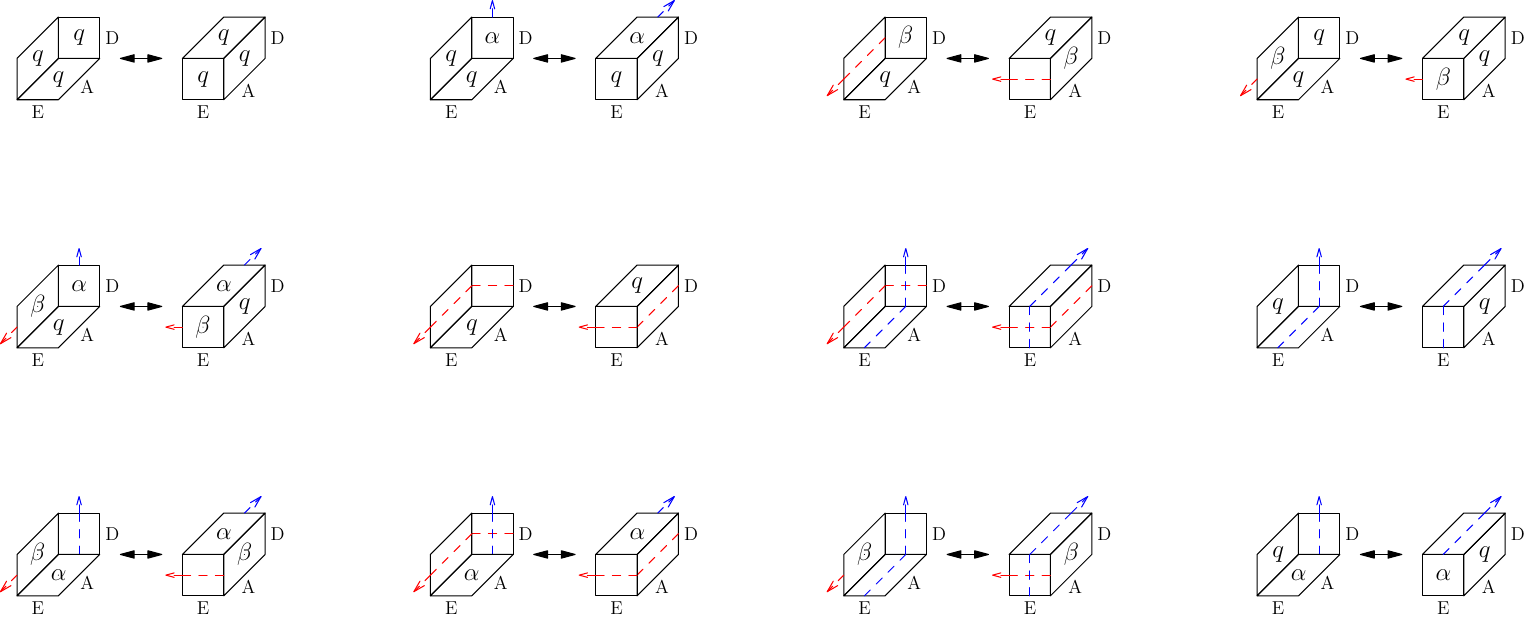}
\caption{The involution $\phi$ from each possible filling of a minimal hexagon (left) to a maximal hexagon (right). The arrows imply compatibility requirements.}
\noindent
\label{flip_map}
\end{figure}

In fact, we can define equivalence classes of rhombic alternative tableaux with the following definitions.

\begin{defn}
A \textbf{weight-preserving flip} on a RAT with tiling $\mathcal{T}$ is the transformation $\phi$ (or the inverse of the $\phi$) given by Figure \ref{flip_map} on some hexagon $\mathfrak{h}$ of $\mathcal{T}$ and the symbols contained in it, while preserving the filling of $\mathcal{T}\backslash\mathfrak{h}$.
\end{defn}

\begin{defn}
Let $W$ be a word in $\{\D, \A, \E\}$, and let $T_1 \in \fil(W,\mathcal{T}_1)$ and $T_2 \in \fil(W,\mathcal{T}_2)$ be rhombic alternative tableaux of type $W$ for arbitrary tilings $\mathcal{T}_1$ and $\mathcal{T}_2$ of $\Gamma(W)$. Then $T_1$ and $T_2$ are \textbf{equivalent} if and only if $T_2$ can be obtained from $T_1$ by a series of weight-preserving flips.
\end{defn}

\subsection{Enumerative considerations}

In this subsection, we show that the number of RAT equivalence classes of size $(n,r)$ is ${n \choose r}\frac{(n+1)!}{(r+1)!}$.

\begin{defn}
Define 
\[
\mathcal{Z}_{n,r}(\alpha,\beta,q) = \sum_W \weight(W)
\]
for $W$ ranging over all words in $\{\D, \A, \E\}^n$ with $r$ $\A$'s.
\end{defn}

By convention, let $\mathcal{Z}_{n,n}(\alpha,\beta,q) = 1$.

\begin{thm} 
\begin{equation}\label{partition_fn}
\mathcal{Z}_{n,r}(\alpha,\beta,1) = {n \choose r} \prod_{i=r}^{n-1} (\alpha+\beta+i\alpha\beta).
\end{equation} 
\end{thm}

\begin{proof}
Let $Z_{n,r,k}(\alpha,\beta,q)$ be the weight generating function for the RAT (with the maximal tiling) with exactly $k$ west-strips that do not contain a $\beta$. 

We also define
\begin{equation}\label{aux_fn}
Z_{n,r}(x) = \sum_{k \geq 0} Z_{n,r,k}(\alpha,\beta,1) x^k
\end{equation}
with $\mathcal{Z}_{n,r}(\alpha,\beta,1) = Z_{n,r}(1)$. We claim that 
\begin{equation}\label{x_partition_fn}
Z_{n,r}(x) = {n \choose r} \prod_{i=r}^{n-1}(x\alpha+\beta+i\alpha\beta).
\end{equation}

We will prove Equation \eqref{x_partition_fn} and hence Equation \eqref{partition_fn} by induction on $n$ in terms of the more refined $Z_{n,r,k}$'s.

First, when $n=1$, $Z_{1,0}(x) = x\alpha+\beta$ and $Z_{1,1}=1$ by convention.

Now, we suppose that Equation \eqref{x_partition_fn} holds for any $N \leq n$ and any $r \leq n-1$. (Again, by convention $Z_{n,n}=1$ for all $n$.) We will show that the formula holds as well for $N=n+1$, for all $r \leq n$.

We begin by observing that Equation \eqref{x_partition_fn} satisfies
\begin{equation}\label{Z_recursion}
Z_{n+1,r}(x) = (x\alpha+\beta+r\alpha\beta) Z_{n,r}(x+\beta) + Z_{n,r-1}(x+\beta).
\end{equation}
We now construct a recursion for $Z_{n+1,r,k}$ in terms of the functions $\{Z_{n,r',k'}\}$ by keeping track of the terms after the addition of a new letter $\D$, $\A$, or $\E$ to the end of a word $W \in \{\D, \A, \E\}^n$. For the following, we denote by $T'$ a tableau in $\fil(W,\mathcal{T}_{max}(W))$ (for $\mathcal{T}_{max}(W)$ the maximal tiling of $\Gamma(W)$), and by $T$ a tableau in $\fil(W x,\mathcal{T}_{max}(W x))$ for $x \in \{\D, \A, \E\}$. We consider all possible cases for $W$ and corresponding $T'$ such that the resulting $T \in \fil(W x,\mathcal{T}_{max}(W x))$ has size $(n+1,r)$ and exactly $k$ west-strips that do not contain a $\beta$. 

\begin{figure}
\centering
\includegraphics[width=0.6\textwidth]{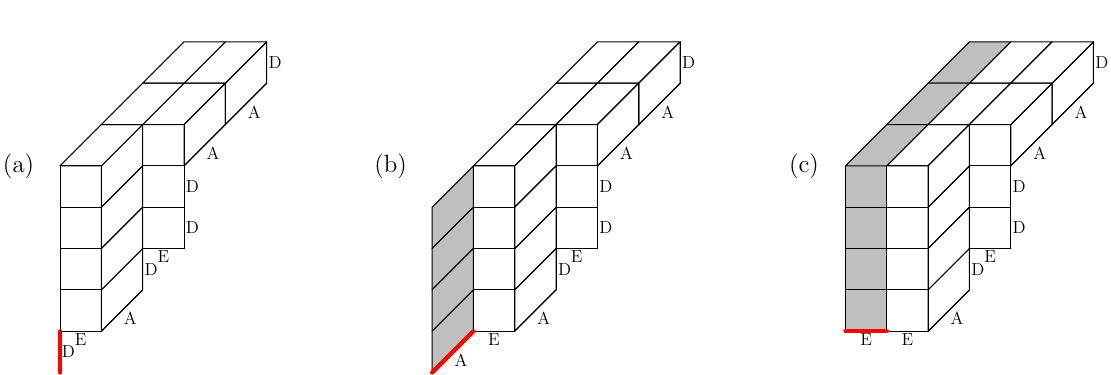}
\caption{Adding a (a) $\D$, (b) $\A$, or (c) $\E$ to the end of $W$.}
\noindent
\label{Zn}
\end{figure}

\begin{enumerate}
\item {\it We add a $\D$ to the end of a PASEP word $W$ of length $n$ with $r$ $\A$'s.} On the tableau level, this corresponds simply to the addition of a D-edge to the southwest end of each $T' \in \fil(W,\mathcal{T}_{max}(W))$ as in Figure \ref{Zn} (a). Since the filling of $T$ is the same as that of $T'$,
\begin{equation}\label{eq_D}
Z_{n+1,r,k} = \alpha Z_{n,r,k-1}.
\end{equation}
This contributes to $\sum_{T \in \fil(W\D,\mathcal{T}_{max}(W\D))}\wt(T)|_{(q=1)}$.

\item {\it We add an $\A$ to the end of a PASEP word $W$ of length $n$ with $r-1$ $\A$'s.} On the tableau level, this corresponds to the addition of a vertical column of some DA tiles to the left boundary of $T'$ to form a tableau with the maximal tiling of $\Gamma(W\A)$ as in Figure \ref{Zn} (b). Suppose $T'$ has $\ell \geq k$ west-strips that do not contain a $\beta$. Then to obtain $T$ with exactly $k$ west-strips that do not contain $\beta$, the $\ell-k$ DA tiles that \emph{do} contain a $\beta$ can be chosen in ${\ell \choose k}$ ways, with the other $k$ DA tiles containing $q$. Therefore, we obtain 
\begin{equation}\label{eq_A}
Z_{n+1,r,k}(\alpha,\beta,1) = \sum_{\ell \geq k} {\ell \choose k} \beta^{\ell-k} Z_{n,r-1,\ell}.
\end{equation}
This contributes to $\sum_{T \in \fil(W\A,\mathcal{T}_{max}(W\A))}\wt(T)|_{(q=1)}$.

\item {\it We add an $\E$ to the end of a PASEP word $W$ of length $n$ with $r$ $\A$'s.} On the tableau level, this corresponds to the addition of a vertical column of some DE tiles followed by a strip of $r$ AE tiles to the left boundary of $T'$ to form a tableau with the maximal tiling of $\Gamma(W\E)$ as in Figure \ref{Zn} (c). Suppose $T'$ has $\ell \geq k$ west-strips that do not contain a $\beta$. We have two cases. 

\begin{enumerate}
\item[(1)] For the first case, there is no $\alpha$ in the newly added DE tiles. Then to obtain $T$ with exactly $k$ west-strips that do not contain $\beta$, the $\ell-k$ DE tiles that \emph{do} contain a $\beta$ can be chosen in ${\ell \choose k}$ ways, with the other $k$ DE tiles containing $q$. Following this, the AE tiles can either contain all $q$'s, or some consecutive string of $q$'s followed by an $\alpha$.

\item[(2)] For the second case, there is an $\alpha$ in the newly added DE tiles, with some $\ell-k \leq u \leq \ell-1$ ``free'' DE tiles below it. (A ``free'' tile means the tile lies in a west-strip that does not contain a $\beta$.) Then, to obtain $T$ with exactly $k$ west-strips that do not contain $\beta$, the $\ell-k$ DE tiles that \emph{do} contain a $\beta$ can be chosen in ${u \choose \ell-k}$ ways, with the other $u-(\ell-k)$ DE tiles that lie below the $\alpha$ containing $q$, and the rest of the tiles empty. 
\end{enumerate}
Combining the above two cases, we obtain
\begin{equation}\label{eq_E}
Z_{n+1,r,k}(\alpha,\beta,1) = \beta \sum_{\ell \geq k} {\ell \choose k} \beta^{\ell-k} (r\alpha+1) Z_{n,r,\ell} + \sum_{u = \ell-k}^{\ell-1} {u \choose \ell-k} \alpha\beta^{\ell-k} Z_{n,r,\ell}.
\end{equation}
This contributes to $\sum_{T \in \fil(W\E,\mathcal{T}_{max}(W\E))}\wt(T)|_{(q=1)}$.
\end{enumerate}

Combining Equations \eqref{eq_D}, \eqref{eq_A}, and \eqref{eq_E} and summing over $k$, we obtain 
\begin{align*}
Z_{n+1,r}(x) &= \sum_{k \geq 0} \left( x\alpha Z_{n,r,k-1} + \sum_{\ell \geq k} {\ell \choose k} \beta^{\ell-k}x^k Z_{n,r-1,\ell} + \beta \sum_{\ell \geq k} {\ell \choose k} \beta^{\ell-k}x^k Z_{n,r,\ell}(r\alpha+1) \right.\\
&\left. + \beta \sum_{\ell \geq k} \sum_{u = \ell-k}^{\ell-1} {u \choose \ell-k} \alpha\beta^{\ell-k} x^k Z_{n,r,\ell}\right)\\
& = x\alpha Z_{n,r}(x) + Z_{n,r-1}(x+\beta)+ \beta(r\alpha+1) Z_{n,r}(x+\beta) + \alpha\beta\sum_{k \geq 0}\sum_{\ell \geq k} {\ell \choose k-1}\beta^{\ell-k}x^k Z_{n,r,\ell}\\
&= x\alpha Z_{n,r}(x) + Z_{n,r-1}(x+\beta) + (r\alpha\beta+\beta+x\alpha) Z_{n,r}(x+\beta) - x\alpha\sum_{k \geq 0} x^{k-1} Z_{n,r,k-1}.
\end{align*}
which simplifies to Equation \eqref{Z_recursion}. Since $Z_{n+1,r}$ satisfies the desired recursion, we thus obtain that Equation \eqref{x_partition_fn} indeed holds for $N=n+1$, and so our proof is complete.
 \end{proof}
 
 \begin{rem} It would be interesting to obtain a bijective proof for Equation \eqref{partition_fn}.
 \end{rem}

\section{Steady state probabilities of the two-species PASEP}\label{ansatz_sec}

Our main result is the following.

\begin{thm}\label{proba}
Let $W$ be a word in $\{\D, \A, \E\}^n$ that represents a state of the two-species PASEP with exactly $r$ ``light'' particles. 
The stationary probability of state $W$ is
\begin{equation}\label{eq_prob}
\Pr(W) = \frac{1}{\mathcal{Z}_{n,r}} \weight(W),
\end{equation}
where $\weight(W)$ is defined in Definition \ref{weight_word}.
\end{thm}

\begin{figure}[h]
\centering
\includegraphics[width=0.8\textwidth]{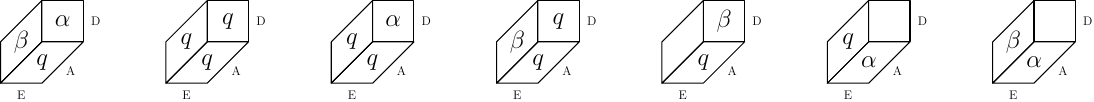}
\caption{All fillings for the minimal tiling of a rhombic diagram of type DAE.}
\noindent
\label{DAE_example}
\end{figure}

\begin{example}
All seven fillings of the minimal tiling of a rhombic diagram corresponding to the word DAE are shown in Figure \ref{DAE_example}. From the sum of the weights of these fillings, we obtain
\[
\Pr(\D\A\E) = \frac{1}{\mathcal{Z}_{3,1}} \left( q^3+\alpha q^2 + \alpha q + \beta q^2 + \beta q + \alpha\beta + \alpha \beta q \right).
\]
\end{example}

To facilitate our proof, we provide a more flexible Matrix Ansatz that generalizes Theorem \ref{ansatz} with the same argument as in an analogous proof for the ordinary PASEP of Corteel and Williams \cite[Theorem 5.2]{CW11}. 

\begin{thm}\label{ansatz2}
Let $f(W)$ be a function on words in the alphabet $\{\D, \A, \E\}$ and $\lambda$ a constant such that for any 
words $X$ and $Y$ in $\{\D, \A, \E\}$, the following conditions are satisfied:
\begin{enumerate}[(I)]
\item $f(X\D\E Y)-q f(X\E\D Y)=\lambda (f( X\D Y)+f(X\E Y))$,
\item $f(X\D\A Y)-q f(X\A\D Y)=\lambda f( X\A Y)$,
\item $f(X\A\E Y)-q f(X\E\A Y)=\lambda f( X\A Y)$,
\item $\beta f(X\D)=\lambda f( X)$,
\item $\alpha f(\E Y)=\lambda f( Y)$.
\end{enumerate} 
Then for any state $W$ of the two-species PASEP of length $n$ with $r$ ``light'' particles,
\[
\Pr(W)= f(W)/Z_{n,r}
\] 
for $Z_{n,r} = \sum_{X}f(X)$ where the sum is over words $X\in\{\D,\A,\E\}^n$ with exactly $r$ $\A$'s.
\end{thm}

\begin{proof}
The proof of Theorem \ref{ansatz2} follows exactly that of \cite[Theorem 5.2]{CW11}.\end{proof}

\begin{rem} Note that the above implies that when the matrix product representation for $D, A, E$ is invoked such that we set $f(X)=\langle w| X |v\rangle$ with $\langle w| |v\rangle=1$, we obtain
\[
Z_{n,r}= [y^r]\langle w| (D+yA+E)^n | v \rangle.
\]
\end{rem}

\subsection{Proof of the main theorem}

\begin{proof}[Proof of \ref{proba}]

Theorem \ref{ansatz2} implies that the steady state probabilities for the two-species PASEP satisfy certain recurrences (that in turn determine all probabilities). The strategy of our proof is to show that the weight generating function for RAT of fixed type satisfies the same recurrences. Specifically, we use these recurrences with the constant $\lambda =\alpha\beta$, to show by induction that for $W$ a word in $\{\D, \A, \E\}^n$ with $r$ $\A$'s, 
\begin{equation}\label{prob}
\weight(W)=\frac{f(W)}{f(\A^r)}.
\end{equation}
The induction is on the area $M$ of our tableaux.

To start, if the area of $\Gamma(W)$ is 0, then necessarily $W=\E^{\ell}\A^r \D^k$ for some $\ell$ and $k$. For every such case, there is a single tiling of $\Gamma(W)$ and a single RAT on that tiling, and both of these are trivial (i.e.\ there are zero tiles to be filled.) We obtain $\weight(\E^{\ell} \A^r \D^k)=\beta^{\ell}\alpha^k$. From Theorem \ref{ansatz2}, we obtain $f(\E^{\ell} \A^r \D^k)=\beta^{\ell}\alpha^k f(\A^r)$, trivially satisfying Equation \eqref{eq_prob}.

Now suppose that any word of area $M<m$ satisfies Equation \eqref{eq_prob}. Let $W$ be a word of length $n$ with $r$ $\A$'s, such that $\Gamma(W)$ has area $m$. Outside of the base case, we assume that at least one of the following must occur:
\begin{enumerate}[i.]
\item $W$ contains an instance of ``DE''.
\item $W$ contains an instance of ``DA''.
\item $W$ contains an instance of ``AE''.
\end{enumerate} 
Based on the occurrence of one of the above, we will express $\weight(W)$ in terms of the weight of some other words whose rhombic diagrams have areas smaller than $m$. Throughout the following, we let $X$ and $Y$ represent some arbitrary words in $\{\D, \A, \E\}$, and we let $T$ represent a RAT of type $W$.

\emph{(i.) $W$ contains an instance of ``$\D\E$''.} We write $W=X\D\E Y$, and suppose $W$ contains $r$ $\A$'s. We can choose an arbitrary tiling $\mathcal{T}$ of $\Gamma(W)$, since any such tiling will contain a DE tile adjacent to the chosen $\D\E$ edges. We call this DE corner tile the \emph{chosen corner}. Let $T \in \fil(W,\mathcal{T})$. The chosen corner of $T$ must contain either an $\alpha$, a $\beta$, or a $q$, so we can decompose the possible fillings of $T$ into three cases. 

If the chosen corner contains an $\alpha$, then all the tiles above it in the same north-strip are empty, and so its entire north-strip has no effect on the rest of the tableau. Thus such $T$ can be mapped to a filling of a smaller RAT on tiling $\mathcal{T}'$ with that north-strip removed, which would have type $X\D Y$ (similar to the example in Figure \ref{ansatz_} (a)). It is easy to check that this operation results in a valid tableau, since any two symbols in the same west-strip of $T$ remain in the same relative position in a west-strip of $T'$. (And similarly for the north-strips, save for the one that was removed). This map gives a bijection between tableaux of type $X\D\E Y$ on tiling $\mathcal{T}$ with an $\alpha$ in the chosen DE corner and tableaux of type $X\D Y$ on tiling $\mathcal{T}'$. The removed column with the $\alpha$ in its bottom-most box has total weight $\alpha\beta$.\footnote{In the total weight of a column, we include the weight of the bottom-most edge, which is a component of the southeast boundary of $T$. When the column removed is an E-column, the weight of the boundary component is $\beta$, so the total weight of the column with an $\alpha$ at the bottom is $\alpha\beta$. Similar reasoning is used in the other cases.} 

Similarly, if the chosen corner contains a $\beta$, then the tiles to its left in the same west-strip must be empty, and so its entire west-strip has no effect on the rest of the tableau. Hence such $T$ can be mapped to a smaller RAT on tiling $\mathcal{T}''$ with that west-strip removed, which would have type $X\E Y$ as in Figure \ref{ansatz_} (b). This map gives a bijection between tableaux of type $X\D\E Y$ on tiling $\mathcal{T}$ with a $\beta$ in the chosen DE corner and tableaux of type $X\E Y$ on tiling $\mathcal{T}''$. The removed west-strip with the $\beta$ in its right-most tile also has total weight $\alpha\beta$. 

Finally, if the chosen corner contains a $q$, then this tile has no effect on the rest of the tableau. Hence such $T$ can be mapped to a RAT of area $m-1$ on tiling $\mathcal{T}'''$ with that DE corner tile removed, which would have type $X\E\D Y$ (similar to the example in Figure \ref{ansatz_} (c)). This map gives a bijection between tableaux of type $X\D\E Y$ on tiling $\mathcal{T}$ with a $q$ in the chosen DE corner and tableaux of type $X\E\D Y$ on tiling $\mathcal{T}'''$. The removed tile with the $q$ has total weight $q$. 

Consequently, we have the sum of the weights of the fillings:
\begin{displaymath} \weight(X\D\E Y) = \weight(X\D Y) \cdot \alpha\beta + \weight(X\E Y) \cdot \alpha\beta + q \weight(X\E\D Y). \end{displaymath}
By the induction hypothesis, since the areas of $\Gamma(X\D Y)$, $\Gamma(X\E Y)$, and $\Gamma(X\E\D Y)$ are all strictly smaller than $m$, we have $\weight(X\D Y) = f(X \D Y)/f(\A^r)$, $\weight(X\E Y) = f(X\E Y)/f(\A^r)$, and $\weight(X\E\D Y) = f(X\E\D Y)/f(\A^r)$. Thus we obtain
\[
(\alpha\beta) \Big(f(X \D Y)+f(X\E Y)+ q f(X \E\D Y)\Big)/f(\A^r) = f(X \D\E Y)/f(\A^r) = f(W)/f(\A^r).
\]
Hence by Theorem \ref{ansatz2} with $\lambda=\alpha\beta$, it follows that $W$ satisfies Equation \eqref{prob}.

\begin{figure}
\centering
\includegraphics[width=\textwidth]{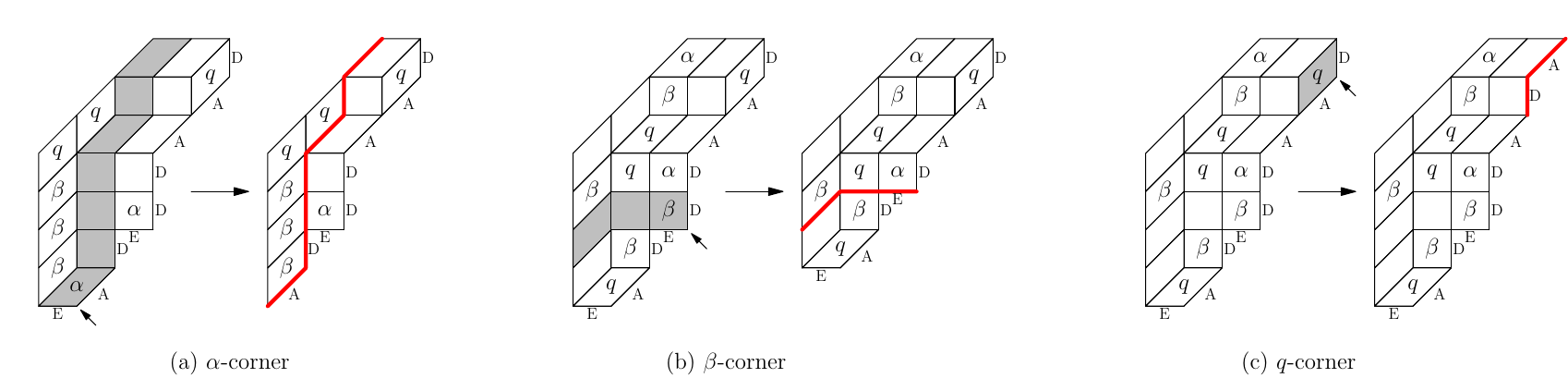}
\caption{(a) $X \A\E Y \mapsto X \A Y$, (b) $X \D\E Y \mapsto X\D Y$, and (c) $X\D\A Y \mapsto q X \A\D Y$.}
\noindent
\label{ansatz_}
\end{figure}

\emph{(ii.) $W$ contains an instance of ``$\A\E$''.} We write $W=X\A\E Y$. We choose tiling $\mathcal{T}$ of $\Gamma(W)$ such that there is an AE tile adjacent to the chosen $\A\E$ edges (and we allow the rest of the tiling to be arbitrary). We call this AE corner tile the \emph{chosen corner}. Let $T \in \fil(W,\mathcal{T})$. The chosen corner of $T$ must contain either an $\alpha$ or a $q$.

If the chosen corner contains an $\alpha$, then the tiles above it in the same north-strip must be empty, and so its entire north-strip has no effect on the rest of the tableau. Hence such $T$ can be mapped to a smaller RAT on tiling $\mathcal{T}'$ with that north-strip removed, which would have type $X\A Y$, as in Figure \ref{ansatz_} (a). This map gives a bijection between tableaux of type $X\A\E Y$ on tiling $\mathcal{T}$ with an $\alpha$ in the chosen AE corner and tableaux of type $X\A Y$ on tiling $\mathcal{T}'$. The removed north-strip with the $\alpha$ in its bottom-most tile has total weight $\alpha\beta$. 

On the other hand, if the chosen corner contains a $q$, then this tile has no effect on the rest of the tableau. Hence such $T$ can be mapped to a RAT of area $m-1$ on tiling $\mathcal{T}''$ with that AE corner tile removed, which would have type $X\E\A Y$ (similar to the example in Figure \ref{ansatz_} (c)). This map gives a bijection between tableaux of type $X\A\E Y$ on tiling $\mathcal{T}$ with a $q$ in the chosen AE corner and tableaux of type $X\E\A Y$ on tiling $\mathcal{T}''$. The removed tile with the $q$ has total weight $q$. Thus we obtain the sum of the weights of the fillings:
\begin{displaymath} \weight(X\A\E Y) = \weight(X\A Y) \cdot \alpha\beta + q \weight(X\E\A Y).\end{displaymath}
Similar reasoning to the DE case completes the argument.

\emph{(iii.) $W$ contains an instance of ``$\D\A$''.} We write $W=X\D\A Y$. We choose tiling $\mathcal{T}$ of $\Gamma(W)$ such that there is a DA tile adjacent to the chosen $\D\A$ edges (and we allow the rest of the tiling to be arbitrary). We call this DA corner tile the \emph{chosen corner}. Let $T \in \fil(W,\mathcal{T})$. The chosen corner of $T$ must contain either a $\beta$ or a $q$.

If the chosen corner contains a $\beta$, then the tiles to its left in the same west-strip must be empty, and so its entire west-strip has no effect on the rest of the tableau. Hence such $T$ can be mapped to a smaller RAT on tiling $\mathcal{T}'$ with that west-strip removed, which would have type $X\A Y$ (similar to the example in Figure \ref{ansatz_} (b)). This map gives a bijection between tableaux of type $X\D\A Y$ on tiling $\mathcal{T}$ with a $\beta$ in the chosen DA corner and tableaux of type $X\A Y$ on tiling $\mathcal{T}'$. The removed west-strip with the $\beta$ in its right-most tile has total weight $\alpha\beta$. 

On the other hand, if the chosen corner contains a $q$, then this tile has no effect on the rest of the tableau. Hence such $T$ can be mapped to a RAT of area $m-1$ on tiling $\mathcal{T}''$ with that DA corner tile removed, which would have type $X\A\D Y$, as in Figure \ref{ansatz_} (c). This map gives a bijection between tableaux of type $X\D\A Y$ on tiling $\mathcal{T}$ with a $q$ in the chosen DA corner and tableaux of type $X\A\D Y$ on tiling $\mathcal{T}''$. The removed tile with the $q$ has total weight $q$. Thus we obtain the sum of the weights of the fillings:
\begin{displaymath} \weight(X\D\A Y) = \weight(X\A Y) \cdot \alpha\beta + q \weight(X\A\D Y).\end{displaymath}
Similar reasoning to the DE case completes the argument.
 
From the above cases, we obtain that for any $M$, any word $W$ with $\Gamma(W)$ of area $M$ satisfies Equation \eqref{prob}, which is the desired result.
\end{proof}

An independent proof of Theorem \ref{proba} can be obtained by constructing a Markov chain on the rhombic alternative tableaux that projects to the two-species PASEP, and is featured in a forthcoming paper by the first author.


\subsection{Complements: the two-species PASEP algebra and the cellular ansatz.} 

It would be possible to give another proof of Theorem \ref{proba} in the spirit of the general theory called \emph{cellular ansatz}, introduced and developed by the second author in \cite{course}.

In the simple case $r=0$ of the PASEP, the Matrix Ansatz defines an algebra with generators $\D$ and $\E$, with relation $\D\E=q\E\D+\E+\D$. In this algebra, any word $W$ with letters $\D$ and $\E$ can be written in a unique way as a sum of monomials $q^t\D^i\E^j$. 

The proof relies on a \emph{planarization} of the \emph{rewriting rules} $\D\E \mapsto q\E\D$, $\D\E \mapsto \E$, and $\D\E \mapsto \D$ (see an example on slides 25-50 of Chapter 3a of \cite{course}). In this context, alternative tableaux appear naturally. We obtain an identity expressing the word $W$ as $\sum_{T}\wt(T)$, where the sum is over alternative tableaux $T$, and $\wt(T)$ is a certain monomial of the form $q^t\D^i\E^j$, where $i$, $j$, and $t$ are defined from the tableau $T$ (see slide 59 of Chapter 3a of \cite{course}). By applying the Matrix Ansatz for the PASEP, we get immediately the interpretation of the stationary probabilities in terms of alternative tableaux (slide 60 of Chapter 3a of \cite{course}).

The general theory of the cellular ansatz works with some family of quadratic algebras $\mathcal{Q}$, having two families of generators, with some commutations relations. Any word $W$ in those generators can be expressed as a sum of monomials over generalized tableaux called \emph{complete} $\mathcal{Q}$-\emph{tableaux}, in bijection with $ \mathcal{Q}$-\emph{tableaux} (see Chapter 6a, slides 41-46, 54-56 of \cite{course}). 

The Matrix Ansatz for the two-species PASEP defines an algebra with three generators $\D, \E, \A$ and three commutation relations 
$\D\E = q\E\D+\E+\D$, $\D\A = q\A\D+\A$, and $\A\E = q\E\A+\A $. This quadratic algebra does not quite fit in the general cellular ansatz theory of Chapter 6a of \cite{course}, but the theory can be extended to such an algebra, by replacing the quadratic lattice by a tiling $\mathcal{T}$ of the diagram $\Gamma(W)$. The corresponding $\mathcal{Q}$-\emph{tableaux} are the RAT and, in a similar way, one can prove that any word $W$ in letters $\{\D, \A, \E\}$ can be expressed in a unique way as a sum of monomials $q^t\D^i\A^m\E^j$.

More precisely we have the following identity
\[
W=\sum_{F \in \fil(W,\mathcal{T})} q^t\E^i\A^m\D^j,
\]
where $i$ (respectively $j$) is the number of north-strips (respectively west-strips) of $F$ not containing an $\alpha$ (respectively $\beta$), and $t$ is the number of cells weighted $q$ as in the definition of $\wt(F)$ in Section \ref{tat_def}. Note that the weight $\wt(F)$ defined in Definition \ref{weight} is equal to the monomial $q^t\alpha^{n-r-i}\beta^{n-r-j}$ where $n$ is the length of $W$ and $r$ is the number of $\A$'s it contains.

Applying to the above the two-species Matrix Ansatz, we obtain immediately Theorem \ref{proba}.


\section{Bijection from RAT with $q=0$ to ``multi-Catalan tableaux''}\label{q0_sec}

The two-species PASEP for the case $q=0$ was studied in a previous paper \cite{M16}, and certain tableaux called ``multi-Catalan tableaux'' were found to give a combinatorial interpretation for the steady state probabilities of this process. A multi-Catalan tableau is a tableau of the same flavor as the alternative tableaux, except with some additional conditions on its rows and columns, which can be labeled as a D-row or A-row and as an E-column or A-column respectively. More precisely, the multi-Catalan tableaux are defined as follows.

\begin{defn} \label{MCT_def}
A \emph{multi-Catalan tableau} of \textbf{size} $(n,r,k)$ is a Young diagram $Y=Y(T)$ with at least $r$ inner corners (i.e. consecutive west- and south-edges on the southwest boundary), that is justified to the northeast and contained in a rectangle of size $(k+r) \times (n-k)$. $Y$ is identified with the lattice path $L=L(T)$ that takes the steps south and west and follows the southeast boundary of $Y$. In addition, we have the following: 
\begin{itemize}
\item Each boundary edge of $L$ is labelled with a D, E, or A such that exactly $r$ inner corners have both edges labeled with $\A$'s, and the remaining west edges have the label E, and the remaining south edges have the label D. 
\item An E-column (respectively A-column) is a column with an E (respectively A) labeling its bottom-most edge.
\item A D-row (respectively A-row) is a row with a D (respectively A) labeling its right-most edge.
\item A DE box is a box in a D-row and an E-column. (The DA, AE, AA boxes are defined correspondingly.)
\end{itemize}
Finally, we fill $T$ with $\alpha$'s and $\beta$'s according to the following rules:
\begin{enumerate}[i.]
\item A box in the same row and left of a $\beta$ must be empty.
\item A box in the same column and above of an $\alpha$ must be empty.
\item A DE box that is not forced to be empty must contain an $\alpha$ or a $\beta$.
\item A DA box that is not forced to be empty must contain a $\beta$.
\item An AE box that is not forced to be empty must contain an $\alpha$.
\end{enumerate}
\end{defn}

\vspace{0.3in}
\begin{rem} 
It is easy to check that any box of the multi-Catalan tableau that has some A-row below it or some A-column to its left must be empty.
\end{rem} 

\begin{figure}[h]
\centering
\includegraphics[width=\textwidth]{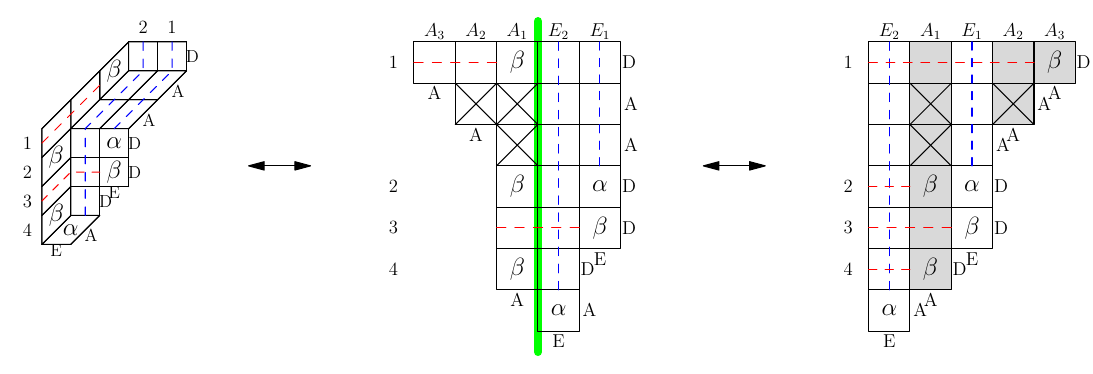}
\caption{Bijection from a RAT of type DAADDEDAE with the minimal tiling $\mathcal{T}_{min}$ to the ``unraveled'' tableau to a multi-Catalan tableau of the same type. Here the north-strips of the RAT correspond to the E-columns of the multi-Catalan tableau, and the west-strips of the RAT correspond to the D-rows of the multi-Catalan tableau. The grey columns of the multi-Catalan tableau are the A-columns that were taken from the left of the ``unraveled'' tableau, and placed in the appropriate location in between the E-columns of the multi-Catalan tableau.}
\noindent
\label{MCT_bijection}
\end{figure}

In Definition \ref{MCT_def}, when we refer to the symbol that a box ``sees'' to its right or below, we mean the first symbol encountered in the same row or column, respectively. For example, in the third tableau of Figure \ref{MCT_bijection}, $\beta$ is the first symbol that the $\beta$ in the row labeled ``2'' sees below it.

\begin{defn}\label{weight}
We call the \emph{weight} $\wt_{MCT}(T)$ of a multi-Catalan tableau $T$ the product of all the $\alpha$'s, $\beta$'s it contains times $\alpha^k \beta^{n-k-r}$ where $(n,r,k)$ is the size of the tableau.
\end{defn}

\begin{defn}
The \emph{type} of the tableau $T$ is the word in $\{\D, \A, \E\}$ that is obtained by reading the southeast boundary of $L(T)$ from northeast to southwest, and reading a $\D$ for each D-labeled south edge, an $\E$ for each E-labeled west edge, and an $\A$ for each ``inner corner'' with a pair of A-labelled edges.
\end{defn}

Let $\MCT(W)$ be the set of multi-Catalan tableaux of type $W$, and $\MCT(n,r)$ be the set of multi-Catalan tableaux of size $(n,r)$. Then the following result holds for the multi-Catalan tableaux.

\begin{thm}[\cite{M16}]
Let $W$ be a word in $\{\D, \A, \E\}^n$ that represents a state of the two-species PASEP for $q=0$ with exactly $r$ ``light'' particles. Let 
\[
\mathcal{Z}^0_{n,r} = \sum_{T \in \MCT(n,r)} \wt_{MCT}(T)
\]
Then the stationary probability of state $W$ is
\[
\Pr(W) = \frac{1}{\mathcal{Z}^0_{n,r}} \sum_{T \in \MCT(W)} \wt_{MCT}(T).
\]
\end{thm}

We now provide the connection from RAT with $q=0$ to the multi-Catalan tableaux.

\begin{prop}
Let $W$ be a word in $\{\D, \A, \E\}$ that represents a state of the two-species PASEP. The multi-Catalan tableaux of type $W$ are in bijection with the rhombic alternative tableaux with $q=0$ on $\Gamma(W)$ with some arbitrary fixed tiling. 
\end{prop}

\begin{proof}
First, it suffices to prove the Proposition for any one fixed tiling, so we will do so for the minimal tiling $\mathcal{T}_{min}$ of $\Gamma(W)$. Our bijection follows the diagram in Figure \ref{MCT_bijection}. In particular, after labeling the north-strips from right to left, we perform the following: 
\begin{enumerate}
\item ``Unravel'' $\mathcal{T}_{min}$ as in Figure \ref{MCT_bijection} to obtain an ``unraveled'' tableau with its rows labeled D or A corresponding to whether the tiles in those rows came from west-strips or not (and similarly, its columns labeled E or A corresponding to whether the tiles in those rows came from north-strips or not.) Keep all $\alpha$'s and $\beta$'s in the same relative locations after unraveling. 
\item Label the A-columns with $A_1,A_2,\ldots$ in the unraveled tableau from right to left. Label the E-columns in the unraveled tableau with $E_1,E_2,\ldots$ from right to left. Add boxes with X's through them (i.e.\ place-holders) to correspond to squares that are both in an A-row and an A-column.
\item For $i=1,2,\ldots$, take the $A_i$ column and place it between the $E_j$ and $E_{j+1}$ column where if the word $W$ is read from right to left, the $i$'th A is between the $j$'th and $j+1$'st E's. For each row, if there is a $\beta$ in the $A_i$ column, move that $\beta$ to the right-most possible empty box in the multi-Catalan tableau.
\end{enumerate}

We claim that the above operations result in a multi-Catalan tableau, as seen in Figure \ref{MCT_bijection}. We keep track of the following quantities:
\begin{itemize}
\item The length of column $E_j$ (respectively $A_i$) is the number of $\D$'s and $\A$'s in $W$ preceding the $j$'th $\E$ (respectively $\A$).
\item The $i$'th D-row is composed of $e_i$ DE boxes followed by $a_i$ DA boxes (from right to left), where $e_i$ is the number of $\E$'s and $a_i$ is the number of $\A$'s following the $i$'th $\D$ in $W$.
\end{itemize} 

First, we show that the map we defined results in a partition with its southeast boundary edges labeled D, A, and E as needed according to Definition \ref{MCT_def}. In particular, we note that every A-column in the unraveled tableau has an A-row immediately below its last box since the number of DA boxes in consecutive D-rows is equal. This implies that whenever an A-column is inserted into the multi-Catalan tableau, the west A-edge on the boundary necessarily has a south A-edge adjacent to its left, and so ends up forming an A-labelled inner corner, as required for the labeling of the A-edges. Furthermore, it is clear that we obtain a partition after inserting the A-columns due to the interpretation of the lengths of the E-columns and the A-columns.

It remains to show that the three conditions below are true for the multi-Catalan tableau if and only if they are true for the RAT it was mapped from. These are easily verified by carefully considering the rules for the fillings of both the RAT and the multi-Catalan tableaux, as seen in the example in Figure \ref{MCT_bijection}.
\begin{enumerate}
\item For any $\alpha$, every tile above it in the same column is empty, there is no $\beta$ to its right in the same row, and no $\alpha$ below it in the same column.
\item Every $\beta$ ends up in a DE or DA box, every tile to its left in the same row is empty, and there is no $\beta$ to its right in the same row, and no $\alpha$ below it in the same column.
\item Any empty box must have either an $\alpha$ below it in the same column or a $\beta$ to its right in the same row.
\end{enumerate}

We have thus a weight-preserving bijection from a RAT with the minimal tiling to a multi-Catalan tableau, and since all tilings of RAT are in bijection with each other, the Proposition follows.
\end{proof}

Finally, for the multi-Catalan tableaux, which translate to results for RAT with $q=0$, there are some more refined enumerative results from \cite{M16} and also \cite{ALS09, DS05}.
\begin{thm}
The weight generating function for the multi-Catalan tableaux of size $n$ and whose type has $r$ A's is
\[
\mathcal{Z}^0_{n,r}(\alpha,\beta,0) = (\alpha\beta)^{n-r} \sum_{p=1}^{n-r} \frac{2r+p}{2n-p} {2n-p \choose n+r} \frac{\alpha^{-p-1}-\beta^{-p-1}}{\alpha^{-1}-\beta^{-1}}.
\]
\end{thm}

\begin{thm}
 The number of multi-Catalan tableaux of size $n$ and whose type has $r$ $\A$'s is
\[
\mathcal{Z}^0_{n,r}(1,1,0) = \frac{2(r+1)}{n+r+2}{2n+1 \choose n-r}.
\]
\end{thm}

\begin{thm}
Let $n \vcentcolon= r+k+\ell$. The number of multi-Catalan tableaux of size $n$ and whose type has $r$ $\A$'s and $k$ $\D$'s is 
\[
\frac{r+1}{n+1}{n+1 \choose k}{n+1 \choose \ell}.
\]
\end{thm}

\noindent \textbf{Acknowledgement.} The first author is very grateful to Lauren Williams and Sylvie Corteel for their mentorship and many useful conversations. The first author also thanks LIAFA at Paris Diderot for their hospitality during the creation of this article, as well as the Chateaubriand Fellowship awarded by the Embassy of France in the United States, the Fondation Sciences Math\'{e}matiques de Paris, the France-Berkeley Fund, the NSF grant DMS-1049513, and the NSF grant DMS-1704874 that supported this work.


\begin{thebibliography}{99}


\bibitem{ALS09} A.~Ayyer, J.~L.~Leibowitz, E.~R.~Speer, On the two species asymmetric exclusion process with semi-permeable boundaries, J. Stat. Phys. 135, no. 5-6, 1009--1037 (2009).

\bibitem{cw_mc} S.\ Corteel and L.\ Williams, A Markov chain on permutations which projects to the PASEP. Int.\ Math.\ Res.\ Not., (2007).

\bibitem{CW11} S.\ Corteel and L.\ Williams, Tableaux combinatorics for the asymmetric exclusion process and Askey-Wilson polynomials. Duke Math.\ J., 159: 385--415, (2011).


\bibitem{DEHP93} B.~Derrida, M.~Evans, V.~Hakim, V.~Pasquier, Exact solution of a 1D asymmetric exclusion model using a matrix formulation, J. Phys. A: Math. Gen. 26, 1493--1517 (1993).

\bibitem{DS05} E.~Duchi, G.~Schaeffer. A combinatorial approach to jumping particles. J. Combin. Theory Ser. A 110, no. 1, 1--29 (2005).

\bibitem{MGP68} J.\ MacDonald, J.\ Gibbs, A.\ Pipkin, Kinetics of biopolymerization on nucleic acid templates,
Biopolymers, 6 issue 1 (1968).

\bibitem{M16} O.~Mandelshtam, Multi-Catalan Tableaux and the Two-Species TASEP, Annales de l'Institut Henri Poincar\'e D, Issue 3 (2016).

\bibitem{Spi70} F.\ Spitzer, Interaction of Markov processes, Adv. Math. 5, 246--290 (1970).

\bibitem{Uch07} M. Uchiyama, Two-Species Asymmetric Simple Exclusion Process with Open Boundaries. Math Stat.\ Mech., (2007).

\bibitem{slides} X.\,G.~Viennot, Alternative tableaux, permutations and partially asymmetric exclusion process. Workshop ``Statistical Mechanics and Quantum-Field Theory Methods in Combinatorial Enumeration'', Isaac Newton Institute for Mathematical Science, Cambridge, 23 April 2008, video and slides available at: \url{http://sms.cam.ac.uk/media/1004}. 

\bibitem{course} X.\,G.~Viennot, Algebraic combinatorics and interactions: the cellular ansatz. Course given at IIT Bombay, January-February 2013, slides available at: \url{http://cours.xavierviennot.org/IIT_Bombay_2013.html}

\end{thebibliography}
\end{document}